\documentclass[12pt]{article}
\usepackage{amsmath,amssymb,amsfonts}
\usepackage[english]{babel}
\usepackage{graphicx}
\usepackage{amsthm}
\usepackage{mathtools}
\usepackage{ctable}
\usepackage{latexsym}
\usepackage{color}
\usepackage{amsfonts}
\usepackage{graphicx}
\newcommand{\Rset}{\mathbb{R}}

\newcommand{\E}{{\mathbb{E}}} 

\newtheorem{theorem}{Theorem}[section]

\newtheorem{lemma}{Lemma}[section]
\newtheorem{proposition}{Proposition}[section]

\theoremstyle{definition}
\newtheorem{definition}{Definition}
\newtheorem{remark}{Remark}
\newtheorem{example}{Example}

\title{\huge \bf On Weighted Extropies 
}
\begin{document}
{\center\huge \bf On Weighted Extropies}
\vskip 3mm

\vskip 5mm
\noindent Narayanaswamy Balakrishnan$^1$, Francesco Buono$^2$ and Maria Longobardi$^3$

\noindent  $^1$ Department of Mathematics and Statistics

\noindent \ \ McMaster University, Hamilton, Ontario, Canada

\noindent \ \ bala@mcmaster.ca

\noindent $^2$ Dipartimento di Matematica e Applicazioni "Renato Caccioppoli"

\noindent \ \ Università degli Studi di Napoli Federico II, Naples, Italy

\noindent \ \ francesco.buono3@unina.it

\noindent $^3$ Dipartimento di Biologia

\noindent \ \ Università degli Studi di Napoli Federico II, Naples, Italy

\noindent \ \ malongob@unina.it

\vskip 3mm
\noindent {\em Keywords:\/} Extropy; Residual lifetime; Past lifetime; Weighted extropy; Bivariate extropy.

\noindent{\em AMS Subject Classification}: 62N05, 62B10.
\vskip 3mm

\begin{abstract}
The extropy is a measure of information introduced by Lad et al. (2015) as dual to entropy. As the entropy, it is a shift-independent information measure. We introduce here the notion of weighted extropy, a shift-dependent information measure which gives higher weights to large values of observed random variables. We also study the weighted residual and past extropies as weighted versions of extropy for residual and past lifetimes. Bivariate versions extropy and weighted extropy are also provided. Several examples are presented through out to illustrate the various concepts introduced here.
\end{abstract}
\vskip 4mm

\section{Introduction}

\subsection{Hazard and reversed hazard rates}

Let $X$ be a non-negative absolutely continuous random variable with probability density function (pdf) $f$, cumulative distribution function (cdf) $F$ and survival function $\overline F$.  In reliability theory, hazard rate function of $X$ (also known as the force of mortality or the failure rate)  where $X$ is the lifetime of a system or a component has found many key application. In the same way the reversed hazard rate of $X$ has also attracted much attention in the literature. In a certain sense, it is the dual function of hazard rate and it bears some interesting features useful in reliability analysis; see Block and Savits (1998) and Finkelstein (2002). 

We define, for $x$ such that $\overline F(x)>0$, the hazard rate function of $X$ at $x$, $r(x)$, as
\begin{align*}
r(x) &= \lim_{\Delta x\to 0^+}\frac{\mathbb P(x<X\leq x+\Delta x|X>x)}{\Delta x}\\
&= \frac{1}{\overline F(x)}\lim_{\Delta x\to 0^+}\frac{\mathbb P(x<X\leq x+\Delta x)}{\Delta x}=\frac{f(x)}{\overline F(x)}.
\end{align*}
Analogously, we define, for $x$ such that $F(x)>0$, the reversed hazard rate function of $X$ at $x$, $q(x)$, as
\begin{align*}
q(x)&=\lim_{\Delta x\to 0^+}\frac{\mathbb P(x-\Delta x<X\leq x|X\leq x)}{\Delta x}\\
&=\frac{1}{F(x)}\lim_{\Delta x\to 0^+}\frac{\mathbb P(x-\Delta x<X\leq x)}{\Delta x}=\frac{f(x)}{F(x)}.
\end{align*}
The hazard rate $r(x)$ can be interpreted as the rate of instantaneous failure occurring immediately after the time point $x$, given that the unit has survived up to time $x$.
Similarly, the reversed hazard rate $q(x)$ can be interpreted as the rate of instantaneous failure occurring immediately before the time point $x$, given that the unit has not survived time $x$. The hazard rate function and the reversed hazard rate function uniquely determine the distribution of $X$; see Barlow and Proschan (1996) for details.

\subsection{Residual and past lifetimes}

In reliability theory, the residual and past lifetimes have a great importance. The residual lifetime of $X$ at time $t$ is defined as $X_t=(X|X>t)$, which is a random variable taking values in $(t,+\infty)$ with pdf $f_{X_t}(x)=\frac{f(x)}{\overline F(t)}$ and survival function $\overline F_{X_t}(x)=\frac{\overline F(x)}{\overline F(t)}$. The past lifetime of $X$ at time $t$ is defined as $_t{}X=(X|X\leq t)$, which is a random variable that takes values in $(0,t]$ with pdf $f_{_t{}X}(x)=\frac{f(x)}{F(t)}$ and distribution function $F_{_t{}X}(x)=\frac{F(x)}{F(t)}$.

\subsection{Entropy, weighted entropy and extropy}

The differential entropy, or Shannon entropy, of a random variable $X$ is a basic concept as a measure of discrimination and information and is defined as
\begin{equation}
H\left(X\right)=-\E\left[\log f(X)\right]=-\int_0^{+\infty}f(x)\log f(x) \mathrm dx,
\end{equation}
where log is the natural logarithm; see Shannon (1948).

Lad et al. (2015) introduced the concept of extropy as dual to entropy and it facilitates the comparison of uncertainties of two random variables. If the extropy of $X$ is less than that of another variable $Y$, then $X$ is said to have more uncertainty than $Y$. For a non-negative absolutely continuous random variable $X$, the extropy is defined as
\begin{equation}
J\left(X\right)=-\frac{1}{2}\E\left[f(X)\right]=-\frac{1}{2}\int_0^{+\infty}f^2(x) \mathrm dx.
\end{equation}

As a measure of information, the Shannon entropy is position-free, i.e., a random variable $X$ has the same Shannon entropy of $X+b$, for any $b\in\Rset$. To avoid this problem, the concept of weighted entropy has been introduced (see Di Crescenzo and Longobardi (2006)) as
\begin{equation}
H^w\left(X\right)=-\E\left[X\log f(X)\right]=-\int_0^{+\infty}xf(x)\log f(x) \mathrm dx.
\end{equation}

To measure the uncertainty about the residual lifetime of $X$ at time $t$, Ebrahimi (1996) introduced the residual entropy as
\begin{equation}
H(X_t)=-\int_t^{+\infty}\frac{f(x)}{\overline F(t)}\log\frac{f(x)}{\overline F(t)}\mathrm dx,
\end{equation}
and it is the differential entropy of the residual lifetime $X_t$. It is also possible to study the uncertainty about the past lifetime [Di Crescenzo and Longobardi (2002)] defined by
\begin{equation}
H(_{t}X)=-\int_0^{t}\frac{f(x)}{F(t)}\log\frac{f(x)}{F(t)}\mathrm dx,
\end{equation}
and it is the differential entropy of the past lifetime $_t{}X$.

Analogous to residual entropy, Qiu (2017) defined the extropy for residual lifetime $X_t$, called the residual extropy at time $t$, as
\begin{equation}
J\left(X_t\right)=-\frac{1}{2}\int_0^{+\infty}f_{X_t}^2(x) \mathrm dx=-\frac{1}{2\overline F^2(t)}\int_t^{+\infty}f^2(x) \mathrm dx.
\end{equation}
In analogous manner, Krishnan et al. (2020) and Kamari and Buono (2020) studied the past extropy defined as
\begin{equation}
J\left(_{t}X\right)=-\frac{1}{2}\int_0^{+\infty}f_{_{t}X}^2(x) \mathrm dx=-\frac{1}{2 F^2(t)}\int_0^{t}f^2(x) \mathrm dx.
\end{equation}

Di Crescenzo and Longobardi (2006) discussed weighted versions of the residual and past entropies. The weighted residual entropy is defined as
\begin{equation}
H^w(X_t)=-\int_t^{+\infty}x\frac{f(x)}{\overline F(t)}\log\frac{f(x)}{\overline F(t)}\mathrm dx,
\end{equation}
while the weighted past entropy is defined as
\begin{equation}
H^w(_{t}X)=-\int_0^{t}x\frac{f(x)}{F(t)}\log\frac{f(x)}{F(t)}\mathrm dx.
\end{equation}

Asadi and Zohrevand (2007) and Di Crescenzo and Longobardi (2009a) introduced the dynamic cumulative residual entropy and the dynamic cumulative entropy, respectively. Sathar and Nair (2019) introduced and studied the dynamic survival extropy $J_s\left(X_t\right)$ defined as
\begin{equation}
\label{eq13}
J_s\left(X_t\right)=-\frac{1}{2\overline F^2(t)}\int_t^{+\infty}\overline F^2(x) \mathrm dx.
\end{equation}

Recently, various authors have discuessed different versions of entropy and extropy and their applications, see, for example, Rao et al. (2004) for cumulative residual entropy, Di Crescenzo and Longobardi (2009b, 2013) for cumulative entropy, and Jahanshani et al. (2019) for cumulative residual extropy.

\subsection{Scope of this paper}

The rest of this paper is organized as follows. In Section 2, we introduce and study the weighted extropy analyzing some of its properties and also present some examples. In Section 3, we introduce the bivariate version of extropy and its weighted form. In section 4, we define and study the weighted residual and past extropies; we present some characterization results and also bounds for these measures of information. Finally, in Section 5, some concluding remarks are made.

\section{Weighted extropy}

Analogous to the weighted entropy, we introduce the concept of weighted extropy as
\begin{equation}
\label{eq12}
J^w\left(X\right)=-\frac{1}{2}\E\left[Xf(X)\right]=-\frac{1}{2}\int_0^{+\infty}xf^2(x) \mathrm dx,
\end{equation}
which can also be alternatively expressed as
\begin{equation}
J^w\left(X\right)=-\frac{1}{2}\int_0^{+\infty}f^2(x)\int_0^x \mathrm dy \mathrm dx=-\frac{1}{2}\int_0^{+\infty} \mathrm dy \int_y^{+\infty}f^2(x) \mathrm dx.
\end{equation}

We now present two examples of distributions with the same extropy, but different weighted extropy. In the first example, we note that the weighted extropy is indeed shift-dependent.
\begin{example}
\label{ex1}
Let $X$ and $Y$ be two random variables such that $X\sim U(0,b)$, $Y\sim U(a,a+b)$, where $a,b>0$. We have $f_X(x)=\frac{1}{b}$, for $x\in(0,b)$, and $f_Y(x)=\frac{1}{b}$, for $x\in(a,a+b)$, and then
\begin{align*}
J(X)&=-\frac{1}{2}\int_0^{b} \frac{1}{b^2} \mathrm dx= -\frac{1}{2b},\\
J(Y)&=-\frac{1}{2}\int_a^{a+b} \frac{1}{b^2} \mathrm dx= -\frac{1}{2b},
\end{align*}
i.e., $X$ and $Y$ have the same extropy, but they have different weighted extropy as seen below:
\begin{align*}
J^w(X)&=-\frac{1}{2}\int_0^{b} x \frac{1}{b^2} \mathrm dx= -\frac{1}{2b^2}\frac{b^2}{2}=-\frac{1}{4}, \\
J^w(Y)&=-\frac{1}{2}\int_a^{a+b} x \frac{1}{b^2} \mathrm dx= -\frac{1}{2b^2}\frac{(a+b)^2-a^2}{2}=-\frac{b^2+2ab}{4b^2}=-\frac{b+2a}{4b},
\end{align*}
and so if $a\ne0$, i.e., $X$ and $Y$ are not identically distributed, then $J^w(X)\ne J^w(Y)$.
\end{example}

\begin{example}
Let $X$ be a random variable with piece-wise constant pdf
\begin{equation*}
f(x)=\sum_{k=1}^n c_k \mathbf1_{[k-1,k)}(x),
\end{equation*}
where $c_k\ge0$, $k=1,\dots,n$, $\sum_{k=1}^n c_k=1$ and $\mathbf1_{[k-1,k)}(x)$ is the indicator function of $x$ on the interval $[k-1,k)$. Then, the extropy and the weighted extropy of $X$ are
\begin{align*}
J(X)&=-\frac{1}{2}\int_0^{n} \sum_{k=1}^n c_k^2 \mathbf1_{[k-1,k)}(x) \mathrm dx= -\frac{1}{2}\sum_{k=1}^n \int_{k-1}^k c_k^2 \mathrm dx= -\frac{1}{2}\sum_{k=1}^n c_k^2, \\
J^w(X)&=-\frac{1}{2}\int_0^{n} x \sum_{k=1}^n c_k^2 \mathbf1_{[k-1,k)}(x) \mathrm dx= -\frac{1}{2}\sum_{k=1}^n \int_{k-1}^k x c_k^2 \mathrm dx \\
&= -\frac{1}{2}\sum_{k=1}^n c_k^2 \frac{k^2-(k-1)^2}{2}=-\frac{1}{4}\sum_{k=1}^n c_k^2(2k-1).
\end{align*}
Because we obtain different distributions through a permutation of $c_1,\dots,c_n$, we observe that they have the same extropy, but different weighted extropy (except in special cases).
\end{example}

We now present an example of distributions with the same weighted extropy, but different extropy.
\begin{example}
\label{ex2}
Let $X$ be a random variable such that $X\sim U(0,b)$, with $b>0$. In Example~\ref{ex1} earlier, we observed that $J(X)=-\frac{1}{2b}$ and $J^w(X)=-\frac{1}{4}$, and so the extropy depends on $b$ while the weighted extropy does not. Thus, if we consider $Y\sim U(0,a)$, with $a>0$ and $a\ne b$, we have random variables $X$ and $Y$ with the same weighted extropy, but different extropy.
\end{example}

Let us now evaluate the weighted extropy of some random variables.

\begin{example}
\begin{itemize}
\item[(a)] Suppose $X$ is exponentially distributed with parameter $\lambda$. Then,
\begin{equation*}
J^w(X)=-\frac{1}{2}\int_0^{+\infty}x\lambda^2 \mathrm e^{-2\lambda x} \mathrm dx=\frac{1}{2}\left[\left.\frac{\lambda x}{2}\mathrm e^{-2\lambda x}\right|_0^{+\infty}-\int_0^{+\infty}\frac{\lambda}{2}\mathrm e^{-2\lambda x}\mathrm dx\right]=-\frac{1}{8}.
\end{equation*}
\item[(b)] Suppose $X$ is uniformly distributed over $(a,b)$. Then,
\begin{equation*}
J^w(X)=-\frac{1}{2}\int_a^b x\frac{1}{(b-a)^2} \mathrm dx= -\frac{1}{2(b-a)^2}\frac{b^2-a^2}{2}=-\frac{1}{4}\frac{b+a}{b-a}.
\end{equation*}
Observe that in this case the weighted extropy can be expressed as the product
\begin{equation*}
J^w(X)=J(X)\mathbb E(X),
\end{equation*}
where $\mathbb E(X)=\frac{a+b}{2}$ and $J(X)=-\frac{1}{2(b-a)}$. Then, $J(X)\leq J^w(X)$ if, and only if, $\mathbb E(X)\leq1$, due to the fact that extropy and weighted extropy are non-positive.

\noindent Let us now look for values of $a$ and $b$ such that the weighted extropy for $Exp(\lambda)$ and $U(a,b)$ are the same. We then have
\begin{eqnarray*}
-\frac{1}{8}=-\frac{1}{4}\frac{b+a}{b-a} \Longleftrightarrow 2(b-a)=b+a \Longleftrightarrow b=3a.
\end{eqnarray*}
Then, $Exp(\lambda)$ and $U(1,3)$ have the same weighted extropy.
\item[(c)] Suppose $X$ is gamma distributed with parameters $\alpha$ and $\beta$, and with pdf
\begin{equation*}
f(x)=\begin{cases} \frac{x^{\alpha-1}\mathrm e^{-x/\beta}}{\beta^{\alpha}\Gamma(\alpha)}, & \mbox{if }x>0 \\ 0, & \mbox{otherwise.}
\end{cases}
\end{equation*}
Then, we have
\begin{eqnarray*}
J^w(X)&=&-\frac{1}{2}\int_0^{+\infty}x \frac{x^{2\alpha-2}\mathrm e^{-2x/\beta}}{\beta^{2\alpha}\Gamma^2(\alpha)}\mathrm dx\\
&=&-\frac{1}{2\beta^{2\alpha}\Gamma^2(\alpha)}\int_0^{+\infty}x^{2\alpha-1}\mathrm e^{-2x/\beta}\mathrm dx \\
&=&-\frac{1}{2^{2\alpha+1}}\frac{\Gamma(2\alpha)}{\Gamma^2(\alpha)},
\end{eqnarray*}
which is free of $\beta$.
\item[(d)] Suppose $X$ is beta distributed with parameters $\alpha$ and $\beta$, and with pdf
\begin{equation*}
f(x)=\begin{cases} \frac{x^{\alpha-1}(1-x)^{\beta-1}}{B(\alpha,\beta)}, & \mbox{if }0<x<1 \\ 0, & \mbox{otherwise,}
\end{cases}
\end{equation*}
where
\begin{equation*}
B(\alpha,\beta)=\int_0^1 x^{\alpha-1}(1-x)^{\beta-1}\mathrm dx=\frac{\Gamma(\alpha)\Gamma(\beta)}{\Gamma(\alpha+\beta)}
\end{equation*}
is the complete beta function. Then, we have
\begin{eqnarray*}
J^w(X)&=&-\frac{1}{2}\int_0^{1}x \frac{x^{2\alpha-2}(1-x)^{2\beta-2}}{B^2(\alpha,\beta)}\mathrm dx\\
&=&-\frac{1}{2}\frac{B(2\alpha,2\beta-1)}{B^2(\alpha,\beta)}
\end{eqnarray*}
if $2\beta-1>0$, i.e., $\beta>\frac{1}{2}$, but if $0<\beta\leq\frac{1}{2}$ we have $J^w(X)=-\infty$.
\end{itemize}
\end{example}

\begin{remark}
Let us now focus our attention on the summability of $xf^2(x)$ on the support of $X$. If the support is unbounded, i.e., $(a,+\infty)$, with $a\ge0$, and the function is bounded, we have to investigate about the summability at infinity. Because $\int_a^{+\infty}f(x)\mathrm dx=1$, we have $\lim_{x\to+\infty}f(x)=0$ and $f(x)$ is infinitesimal of higher order with respect to $\frac{1}{x^{1+\varepsilon}}$, for $x\rightarrow+\infty$, for some $\varepsilon>0$. Then, $xf^2(x)$ is infinitesimal of higher order with respect to $\frac{1}{x^{1+2\varepsilon}}$, for $x\rightarrow+\infty$ and so it is integrable, i.e., the weighted extropy is finite. If the support and the density are unbounded, we also have to study the summability at $a$. Suppose $a>0$. If $\lim_{x\to a^+}f(x)=+\infty$, by the normalization condition we know that $f(x)$ is infinity of lower order with respect to $\frac{1}{(x-a)^{1-\varepsilon}}$, for $x\rightarrow a^+$, for some $0<\varepsilon<1$. Hence, $xf^2(x)$ is infinity of lower order with respect to $\frac{1}{(x-a)^{2-2\varepsilon}}$, for $x\rightarrow a^+$, and so is integrable if $\varepsilon\in\left(\frac{1}{2},1\right)$. If $a=0$, by the normalization condition, we know that $f(x)$ is infinity of lower order with respect to $\frac{1}{x^{1-\varepsilon}}$, for $x\rightarrow 0^+$, for some $0<\varepsilon<1$. Hence, $xf^2(x)$ is infinity of lower order with respect to $\frac{1}{x^{1-2\varepsilon}}$, for $x\rightarrow 0^+$, and so is integrable. If the support is bounded and $f$ is bounded, then $xf^2(x)$ is bounded and is integrable. If the support is bounded and $f$ is unbounded, then we can refer to the previous cases. Observe that if the support is $(0,+\infty)$, the weighted extropy is always finite.
\end{remark}

In the following theorem, we study weighted extropy under monotone transformation.

\begin{theorem}
\label{thm2}
Let $Y=\Phi(X)$, with $\Phi$ being strictly monotone, continuous and differentiable, with derivative $\Phi'$. Then, we have
\begin{equation}
\label{eq11}
J^w(Y)=\begin{cases} -\frac{1}{2}\int_{0}^{+\infty}\frac{\Phi(x)}{\Phi'(x)}f_X^2(x)\mathrm dx, & \mbox{if }\Phi\mbox{ is strictly increasing} \\ -\frac{1}{2}\int_0^{+\infty}\frac{\Phi(x)}{\left|\Phi'(x)\right|}f_X^2(x)\mathrm dx, & \mbox{if }\Phi\mbox{ is strictly decreasing.}
\end{cases}
\end{equation}
\end{theorem}

\begin{proof}
From~\eqref{eq12}, we have 
\begin{equation*}
J^w(Y)=-\frac{1}{2}\int_0^{+\infty}x\frac{f_X^2(\Phi^{-1}(x))}{(\Phi'(\Phi^{-1}(x)))^2}\mathrm dx.
\end{equation*}
Let $\Phi$ be strictly increasing. Then, with a change of variable in the above integral, we get
\begin{equation*}
J^w(Y)=-\frac{1}{2}\int_{0}^{+\infty}\frac{\Phi(x)}{\Phi'(x)}f_X^2(x)\mathrm dx,
\end{equation*}
giving the first expression in~\eqref{eq11}. If $\Phi$ is strictly decreasing, we similarly obtain
\begin{equation*}
J^w(Y_t)=-\frac{1}{2}\int_0^{+\infty}\frac{\Phi(x)}{\left|\Phi'(x)\right|}f_X^2(x)\mathrm dx,
\end{equation*}
which is the second expression in~\eqref{eq11}.
\end{proof}

\begin{remark}
If in Theorem \ref{thm2} we take $\Phi(X)=F_X(X)$, then we get
\begin{equation*}
J^w(Y)=-\frac{1}{2}\int_0^{+\infty}F_X(x)f_X(x)\mathrm dx=-\frac{1}{4},
\end{equation*}
which agrees with that of Uniform$(0,1)$ distribution, since it is known in this case that the probability integral transformation $Y=F_X(X)$ is Uniform$(0,1)$.
\end{remark}

In the following theorem, we obtain a lower bound for the weighted extropy of the sum of two independent random variables.

\begin{theorem}
\label{thm3}
Let $X$ and $Y$ be two non-negative independent random variables with densities $f_X$ and $f_Y$, respectively. Then,
\begin{equation}
J^w(X+Y)\ge-2\left(J(X)J^w(Y)+J^w(X)J(Y)\right).
\end{equation}
\end{theorem}

\begin{proof}
Since $X$ and $Y$ are non-negative independent random variables, the density function of $Z=X+Y$ is given, for $z>0$, by
\begin{equation*}
f_Z(z)=\int_0^z f_X(x)f_Y(z-x)\mathrm dx.
\end{equation*}
Hence, the weighted extropy of $Z$ is given by
\begin{equation*}
J^w(Z)=-\frac{1}{2}\int_0^{+\infty}z\left[\int_0^z f_X(x)f_Y(z-x)\mathrm dx\right]^2\mathrm dz.
\end{equation*}
Using Jensen's inequality, we then have
\begin{eqnarray*}
J^w(Z)&\ge& -\frac{1}{2}\int_0^{+\infty}z\int_0^z f_X^2(x)f_Y^2(z-x)\mathrm dx\mathrm dz \\
&=&-\frac{1}{2}\int_0^{+\infty}f_X^2(x) \int_x^{+\infty} z f_Y^2(z-x)\mathrm dz\mathrm dx \\
&=& -\frac{1}{2}\int_0^{+\infty}f_X^2(x) \int_0^{+\infty} (z+x) f_Y^2(z)\mathrm dz\mathrm dx \\
&=& \int_0^{+\infty}f_X^2(x) (J^w(X)+xJ(Y))\mathrm dx \\
&=& -2J(X)J^w(Y)-2J(Y)J^w(X),
\end{eqnarray*}
as required.
\end{proof}

\begin{remark}
In particular, if $X$ and $Y$ are independent and identically distributed in Theorem~\ref{thm3}, we simply deduce
\begin{equation*}
J^w(X+Y)\ge-4J(X)J^w(X).
\end{equation*}
\end{remark}

\section{Bivariate version of extropy and weighted extropy}

It is possible to introduce bivariate version of extropy. If $X$ and $Y$ are non-negative absolutely continuous random variables, the bivariate version of extropy, denoted by $J(X,Y)$, is defined as
\begin{equation}
J(X,Y)=\frac{1}{4}\mathbb E[f(X,Y)]=\frac{1}{4}\int_0^{+\infty}\int_0^{+\infty}f^2(x,y) \mathrm dx \mathrm dy,
\end{equation}
where $f(x,y)$ is the joint density of $(X,Y)$.

\begin{remark}
This measure can also be defined for a general $k$-dimensional vector. In that case, in the definition of $J(X_1,X_2,\dots,X_k)$, the multiplying factor for the integral will be $\left(-\frac{1}{2}\right)^k$.
\end{remark}

In the following proposition, we discuss how extropy changes under linear transformations and how bivariate version of extropy behaves in the case of independence.

\begin{proposition}
\label{pr1}
\begin{itemize}
\item[(i)] Let $X$ be a non-negative absolutely continuous random variable, and $Y=aX+b$, with $a>0, b\ge0$. Then, $J(Y)=\frac{1}{a}J(X)$;
\item[(ii)] Let $X,Y$ be non-negative absolutely continuous random variables. If $X$ and $Y$ are independent, then $J(X,Y)=J(X)J(Y)$.
\end{itemize}
\end{proposition}

\begin{proof}
\begin{itemize}
\item[$\textit{(i)}$] From $Y=aX+b$, we readily have $f_Y(x)=\frac{1}{a}f_X\left(\frac{x-b}{a}\right)$, $x>b$. So,
\begin{equation*}
J(Y)=-\frac{1}{2}\int_b^{+\infty}\frac{1}{a^2} f_X^2\left(\frac{x-b}{a}\right)\mathrm dx=-\frac{1}{2}\int_0^{+\infty}\frac{1}{a}f_X^2(x) \mathrm dx= \frac{1}{a}J(X).
\end{equation*}

\item[$\textit{(ii)}$] If $X$ and $Y$ are independent, then $f(x,y)=f_X(x)f_Y(y)$ and hence
\begin{eqnarray*}
J(X,Y)&=&\frac{1}{4}\mathbb E[f(X,Y)]=\frac{1}{4}\mathbb E[f_X(X)f_Y(Y)] \\
&=&\frac{1}{4}\mathbb E[f_X(X)]\mathbb E[f_Y(Y)]=J(X)J(Y).
\end{eqnarray*}
\end{itemize}
\end{proof}

\begin{remark}
In Property $(i)$ of Proposition \ref{pr1}, if we choose $a=1$, we get the known property that extropy is invariant under translations.
\end{remark}

In a spirit similar to that of the bivariate version of extropy, we can also introduce bivariate weighted extropy as
\begin{equation}
J^w(X,Y)=\frac{1}{4}\mathbb E[XYf(X,Y)]=\frac{1}{4}\int_0^{+\infty}\int_0^{+\infty}xyf^2(x,y) \mathrm dx \mathrm dy,
\end{equation}
where $X$ and $Y$ are non-negative random variables with joint density function $f(x,y)$.

In the following proposition we discuss how weighted extropy behaves under linear transformations and the bivariate extropy in the case of independence.

\begin{proposition}
\label{pr2}
\begin{itemize}
\item[(i)] Let $X$ be a non-negative absolutely continuous random variable, and $Y=aX+b$, with $a>0, b\ge0$. Then, $J^w(Y)=J^w(X)+\frac{b}{a}J(X)$;
\item[(ii)] Let $X,Y$ be non-negative absolutely continuous random variables. If $X$ and $Y$ are independent, then $J^w(X,Y)=J^w(X)J^w(Y)$.
\end{itemize}
\end{proposition}

\begin{proof}
\begin{itemize}
\item[$\textit{(i)}$] From $Y=aX+b$, we have $f_Y(x)=\frac{1}{a}f_X\left(\frac{x-b}{a}\right)$, $x>b$, and so
\begin{eqnarray*}
J^w(Y)&=&-\frac{1}{2}\int_b^{+\infty}x \frac{1}{a^2} f_X^2\left(\frac{x-b}{a}\right)\mathrm dx=-\frac{1}{2}\int_0^{+\infty}(ax+b)\frac{1}{a}f_X^2(x) \mathrm dx \\
&=&-\frac{1}{2}\int_0^{+\infty}xf_X^2(x) \mathrm dx -\frac{1}{2}\frac{b}{a}\int_0^{+\infty}f_X^2(x) \mathrm dx= J^w(X)+\frac{b}{a}J(X).
\end{eqnarray*}

\item[$\textit{(ii)}$] If $X$ and $Y$ are independent, then $f(x,y)=f_X(x)f_Y(y)$ and hence
\begin{eqnarray*}
J^w(X,Y)&=&\frac{1}{4}\mathbb E[XYf(X,Y)]=\frac{1}{4}\mathbb E[XYf_X(X)f_Y(Y)] \\
&=&\frac{1}{4}\mathbb E[Xf_X(X)]\mathbb E[Yf_Y(Y)]=J^w(X)J^w(Y).
\end{eqnarray*}
\end{itemize}
\end{proof}

\begin{remark}
In Part $(i)$ of Proposition \ref{pr2}, if we choose $b=0$, we see that the weighted extropy is invariant for proportional random variables, as we saw earlier in Example~\ref{ex2} in the case of uniform distribution.
\end{remark}

\begin{example}
\begin{itemize}
\item[(a)] Let us consider $(X,Y)$ as a bivariate beta random variable with joint density function
\begin{equation*}
f(x,y)=\frac{1}{B(\alpha,\beta,\gamma)}x^{\alpha-1}(y-x)^{\beta-1}(1-y)^{\gamma-1}, \ 0<x<y<1, \ \alpha,\beta,\gamma>0,
\end{equation*}
where $B(\alpha,\beta,\gamma)=\frac{\Gamma(\alpha)\Gamma(\beta)\Gamma(\gamma)}{\Gamma(\alpha+\beta+\gamma)}$ is the bivariate complete beta function. Then, we readily find the bivariate extropy as
\begin{eqnarray*}
J(X,Y)&=&\frac{1}{4}\mathbb E[f(X,Y)]\\
&=&\frac{1}{4 B^2(\alpha,\beta,\gamma)}\iint_{0<x<y<1}x^{2\alpha-2}(y-x)^{2\beta-2}(1-y)^{2\gamma-2}\mathrm dx \mathrm dy\\
&=&\frac{B(2\alpha-1,2\beta-1,2\gamma-1)}{4 B^2(\alpha,\beta,\gamma)}
\end{eqnarray*}
provided $\alpha,\beta,\gamma>\frac{1}{2}$.
\item[(b)] For the above bivariate beta distribution, we readily find the bivariate weighted extropy to be
\begin{eqnarray*}
J^w(X,Y)&=&\frac{1}{4}\mathbb E[XYf(X,Y)]\\
&=&\frac{1}{4 B^2(\alpha,\beta,\gamma)}\iint_{0<x<y<1}x^{2\alpha-1}y(y-x)^{2\beta-2}(1-y)^{2\gamma-2}\mathrm dx \mathrm dy\\
&=&\frac{1}{4 B^2(\alpha,\beta,\gamma)}\left[B(2\alpha,2\beta,2\gamma-1)+B(2\alpha+1,2\beta-1,2\gamma-1)\right]
\end{eqnarray*}
provided $\alpha>0$ and $\beta,\gamma>\frac{1}{2}$.
\item[(c)] In the special case of bivariate uniform distribution (i.e., $\alpha=\beta=\gamma=1$), the above expressions readily reduce to
\begin{equation*}
J(X,Y)=\frac{1}{2} \ \ \mbox{  and  } \ \ J^w(X,Y)=\frac{1}{8}.
\end{equation*}
\end{itemize}
\end{example}

\section{Weighted residual and past extropies}

In this section, we introduce and study weighted residual extropy and weighted past extropy.
\begin{definition}
Let $X$ be a non-negative absolutely continuous random variable. For all $t$ in the support of $f$, we define
\begin{itemize}
\item[(i)] the weighted residual extropy of $X$ at time $t$ as
\begin{equation}
\label{eq2}
J^w\left(X_t\right)=-\frac{1}{2\overline F^2(t)}\int_t^{+\infty}x f^2(x) \mathrm dx;
\end{equation}
\item[(ii)] the weighted past extropy of $X$ at time $t$ as
\begin{equation}
\label{eq3}
J^w\left(_{t}X\right)=-\frac{1}{2 F^2(t)}\int_0^{t}x f^2(x) \mathrm dx.
\end{equation}
\end{itemize}
\end{definition}

\begin{remark}
We notice that
\begin{equation}
\lim_{t\to 0^+}J^w\left(X_t\right)=\lim_{t\to+\infty}J^w\left(_{t}X\right)=J^w(X).
\end{equation}
\end{remark}

In the following lemma, we evaluate the derivative of the weighted residual extropy and the derivative of the weighted past extropy.
\begin{lemma}
\label{lem1}
Let $X$ be a non-negative absolutely continuous random variable with weighted residual extropy $J^w(X_t)$ and weighted past extropy $J^w(_t{}X)$. Then,
\begin{itemize}
\item[(i)] $\frac{\mathrm d}{\mathrm dt}J^w\left(X_t\right)=\frac{r(t)}{2}\left[J^w\left(X_t\right)+tr(t)\right]$, where $r(t)$ is the hazard rate function of $X$;
\item[(ii)] $\frac{\mathrm d}{\mathrm dt}J^w\left(_{t}X\right)=-\frac{q(t)}{2}\left[J^w\left(_{t}X\right)+tq(t)\right]$, where $q(t)$ is the reversed hazard rate function of $X$.
\end{itemize}
\end{lemma}

\begin{proof}
\begin{itemize}
\item[$\textit{(i)}$] From the definitions of weighted residual extropy and hazard rate function, we have
\begin{eqnarray*}
\frac{\mathrm d}{\mathrm dt}J^w\left(X_t\right)&=&-\frac{1}{4\overline F^3(t)}f(t)\int_t^{+\infty}xf^2(x)\mathrm dx+\frac{1}{2\overline F^2(t)}tf^2(t)\\
&=&\frac{r(t)}{2}\left[J^w\left(X_t\right)+tr(t)\right];
\end{eqnarray*}
\item[$\textit{(ii)}$] From the definition of weighted past extropy and reversed hazard rate function, we have
\begin{eqnarray*}
\frac{\mathrm d}{\mathrm dt}J^w\left(_{t}X\right)&=&\frac{1}{4F^3(t)}f(t)\int_0^{t}xf^2(x)\mathrm dx-\frac{1}{2F^2(t)}tf^2(t)\\
&=&-\frac{q(t)}{2}\left[J^w\left(_{t}X\right)+tq(t)\right].
\end{eqnarray*}
\end{itemize}
\end{proof}

\begin{remark}
We may ask a natural question here whether the weighted residual extropy could be constant over the support of a non-negative absolutely continuous random variable. If $J^w(X_t)$ is constant, then for all $t>0$, we have 
\begin{equation*}
J^w(X_t)+tr(t)=0
\end{equation*}
and then
\begin{equation*}
\int_t^{+\infty}xf^2(x)\mathrm dx=2tf(t)\overline F(t).
\end{equation*}
Differentiating both sides we get
\begin{equation*}
tf^2(t)=2f(t)\overline F(t)+2tf'(t)\overline F(t).
\end{equation*}
We know that $r'(t)=\frac{f'(t)}{\overline F(t)}+r^2(t)$ and so dividing by $\overline F^2(t)$ both sides of the above equality, we obtain
\begin{equation*}
r'(t)=-\frac{r(t)}{t}+\frac{3}{2}r^2(t)
\end{equation*}
which is a Bernoulli differential equation with initial condition $r(t_0)=r_0>0$, $t_0>0$. Upon solving this differential equation, we get
\begin{equation*}
r(t)=\frac{2}{t(C-3\log t)},
\end{equation*}
where $C=\frac{2}{t_0 r_0}+3\log t_0$. We know that the hazard rate function is non-negative and this condition is satisfied if and only if $t\leq\mathrm e^C/3$. Hence, the weighted residual extropy can not be constant over $(0,+\infty)$.
\end{remark}

\begin{remark}
The weighted extropy, the weighted residual extropy and the weighted past extropy satisfy the following relationship:
\begin{equation}
J^w(X)=F^2(t)J^w(_{t}X)+\overline F^2(t)J^w(X_{t}).
\end{equation}
To see this, we consider
\begin{eqnarray*}
J^w(X)&=&-\frac{1}{2}\int_0^{+\infty}xf^2(x)\mathrm dx=-\frac{1}{2}\int_0^{t}xf^2(x)\mathrm dx-\frac{1}{2}\int_t^{+\infty}xf^2(x)\mathrm dx\\
&=&-\frac{1}{2}F^2(t)\int_0^tx\frac{f^2(x)}{F^2(t)}\mathrm dx-\frac{1}{2}\overline F^2(t)\int_t^{+\infty}x\frac{f^2(x)}{\overline F^2(t)}\mathrm dx\\
&=&F^2(t)J^w(_{t}X)+\overline F^2(t)J^w(X_{t}).
\end{eqnarray*}
\end{remark}

\begin{theorem}
If $X$ is a non-negative absolutely continuous random variable and if $J^w(X_t)$ is increasing in $t>0$, then $J^w(X_t)$ uniquely determines the distribution of $X$.
\end{theorem}

\begin{proof}
From Lemma~\ref{lem1}, we have
\begin{equation*}
\frac{\mathrm d}{\mathrm dt}J^w\left(X_t\right)=\frac{r(t)}{2}\left[J^w\left(X_t\right)+tr(t)\right].
\end{equation*}
Let us now introduce the following function:
\begin{equation*}
g(x)=\frac{x}{2}\left[J^w\left(X_t\right)+tx\right]-\frac{\mathrm d}{\mathrm dt}J^w\left(X_t\right).
\end{equation*}
We know that $g(r(t))=0$ and $g(0)=-\frac{\mathrm d}{\mathrm dt}J^w\left(X_t\right)\leq0$ because $J^w(X_t)$ is increasing in $t>0$. Moreover, $\lim_{x\to+\infty}g(x)=+\infty$. If we evaluate the derivative of $g(x)$, we could see that there is only one point at which it vanishes; in fact,
\begin{equation*}
\frac{\mathrm d}{\mathrm dx}g(x)=\frac{1}{2}J^w(X_t)+tx 
\end{equation*}
and so
\begin{equation*}
\frac{\mathrm d}{\mathrm dx}g(x)=0 \Longleftrightarrow x=-\frac{1}{2t}J^w(X_t) \ (\ge0).
\end{equation*}
Then, $g(x)=0$ has a unique solution and it is $r(t)$. From Barlow and Proschan (1996), we know that the hazard rate function uniquely determines the distribution and so $J^w(X_t)$ also uniquely determines the distribution.
\end{proof}

\begin{theorem}
If $X$ is a non-negative absolutely continuous random variable and if $J^w(_tX)$ is decreasing in $t>0$, then $J^w(_tX)$ uniquely determines the distribution of $X$.
\end{theorem}

\begin{proof}
From Lemma~\ref{lem1}, we have
\begin{equation*}
\frac{\mathrm d}{\mathrm dt}J^w\left(_tX\right)=-\frac{q(t)}{2}\left[J^w\left(_tX\right)+tq(t)\right].
\end{equation*}
Let us now introduce the following function:
\begin{equation*}
h(x)=\frac{x}{2}\left[J^w\left(_tX\right)+tx\right]+\frac{\mathrm d}{\mathrm dt}J^w\left(_tX\right).
\end{equation*}
We know that $h(q(t))=0$ and $h(0)=\frac{\mathrm d}{\mathrm dt}J^w\left(_tX\right)\leq0$ because $J^w(_tX)$ is decreasing in $t>0$. Moreover, $\lim_{x\to+\infty}h(x)=+\infty$. If we evaluate the derivative of $h(x)$, we could see that there is only one point at which it vanishes; in fact,
\begin{equation*}
\frac{\mathrm d}{\mathrm dx}h(x)=\frac{1}{2}J^w(_tX)+tx 
\end{equation*}
and so
\begin{equation*}
\frac{\mathrm d}{\mathrm dx}h(x)=0 \Longleftrightarrow x=-\frac{1}{2t}J^w(_tX) \ (\ge0).
\end{equation*}
Then, $h(x)=0$ has a unique solution and it is $q(t)$. From Barlow and Proschan (1996), we know that the reversed hazard rate function uniquely determines the distribution and so $J^w(_tX)$ also uniquely determines the distribution.
\end{proof}

In the following two theorems, we obtain bounds for the weighted residual extropy and the weighted past extropy under the monotonicity of hazard rate and reversed hazard rate functions.

\begin{theorem}
\label{thm1}
If the hazard rate function $r$ is increasing, then 
\begin{equation}
\label{eq1}
J^w(X_t)\leq t\, r^2(t)J_s\left(X_t\right),
\end{equation}
where $J_s\left(X_t\right)$ is the dynamic survival extropy defined in \eqref{eq13}.
\end{theorem}

\begin{proof}
From the definition of the weighted residual extropy, we have
\begin{equation*}
J^w\left(X_t\right)=-\frac{1}{2 \overline F^2(t)}\int_t^{+\infty}x f^2(x) \mathrm dx=-\frac{1}{2\overline F^2(t)}\int_t^{+\infty}x r^2(x)\overline F^2(x) \mathrm dx.
\end{equation*}
Because the hazard rate function is increasing, we have
\begin{eqnarray*}
-\frac{1}{2\overline F^2(t)}\int_t^{+\infty}x r^2(x)\overline F^2(x) \mathrm dx&\leq&-\frac{r^2(t)}{2\overline F^2(t)}\int_t^{+\infty}x \overline F^2(x) \mathrm dx\\
&\leq&-t\frac{r^2(t)}{2\overline F^2(t)}\int_t^{+\infty}\overline F^2(x)\mathrm dx \\
&=&t\, r^2(t)J_s\left(X_t\right).
\end{eqnarray*}
\end{proof}

\begin{theorem}
If the reversed hazard rate function $q$ is increasing in $(0,T)$, with $T>t$, then 
\begin{equation}
J^w(_{t}X)\ge -t\frac{q^2(t)}{2}.
\end{equation}
\end{theorem}

\begin{proof}
From the definition of the weighted past extropy, we have
\begin{equation*}
J^w\left(_{t}X\right)=-\frac{1}{2 F^2(t)}\int_0^{t}x f^2(x) \mathrm dx=-\frac{1}{2 F^2(t)}\int_0^{t}x q^2(x)F^2(x) \mathrm dx.
\end{equation*}
Because the reversed hazard rate function is increasing in $(0,T)$, we have
\begin{equation*}
-\frac{1}{2 F^2(t)}\int_0^{t}x q^2(x)F^2(x) \mathrm dx\ge-\frac{q^2(t)}{2 F^2(t)}\int_0^{t}x F^2(x) \mathrm dx;
\end{equation*}
moreover, the distribution function is increasing and so
\begin{equation*}
-\frac{q^2(t)}{2 F^2(t)}\int_0^{t}x F^2(x) \mathrm dx\ge-\frac{q^2(t)}{2}\int_0^{t}x \mathrm dx=-t\frac{q^2(t)}{2}.
\end{equation*}
\end{proof}

\begin{example}
For the exponential distribution with parameter $\lambda>0$, the hazard function $r(t)=\lambda$ is constant and so satisfies the conditions in Theorem~\ref{thm1}. The weighted residual extropy is given by
\begin{equation*}
J^w(X_t)=-\frac{1}{2\mathrm e^{-2\lambda t}}\int_t^{+\infty}x\lambda^2\mathrm e^{-2\lambda t}\mathrm dx=-\frac{\lambda t}{4}-\frac{1}{8},
\end{equation*}
while the upper bound given by Theorem \ref{thm1} is
\begin{equation*}
-t\frac{\lambda^2}{2\mathrm e^{-2\lambda t}}\int_t^{+\infty}\mathrm e^{-2\lambda x}\mathrm dx=-\frac{\lambda t}{4},
\end{equation*}
and so~\eqref{eq1} is fulfilled.
\end{example}

In the following theorem, we discuss weighted residual extropy and weighted past extropy under monotone transformation.

\begin{theorem}
Let $Y=\Phi(X)$, with $\Phi$ being strictly monotone, continuous and differentiable, with derivative $\Phi'$. Then, for all $t>0$, we have
\begin{equation}
\label{eq4}
J^w(Y_t)=\begin{cases} -\frac{1}{2\overline F_X^2(\Phi^{-1}(t))}\int_{\Phi^{-1}(t)}^{+\infty}\frac{\Phi(x)}{\Phi'(x)}f_X^2(x)\mathrm dx, & \mbox{if }\Phi\mbox{ is strictly increasing} \\ -\frac{1}{2F_X^2(\Phi^{-1}(t))}\int_0^{\Phi^{-1}(t)}\frac{\Phi(x)}{\left|\Phi'(x)\right|}f_X^2(x)\mathrm dx, & \mbox{if }\Phi\mbox{ is strictly decreasing}
\end{cases}
\end{equation}
and
\begin{equation}
\label{eq5}
J^w(_{t}Y)=\begin{cases} -\frac{1}{2F_X^2(\Phi^{-1}(t))}\int_0^{\Phi^{-1}(t)}\frac{\Phi(x)}{\Phi'(x)}f_X^2(x)\mathrm dx, & \mbox{if }\Phi\mbox{ is strictly increasing} \\ -\frac{1}{2\overline F_X^2(\Phi^{-1}(t))}\int_{\Phi^{-1}(t)}^{+\infty}\frac{\Phi(x)}{\left|\Phi'(x)\right|}f_X^2(x)\mathrm dx, & \mbox{if }\Phi\mbox{ is strictly decreasing.}
\end{cases}
\end{equation}
\end{theorem}

\begin{proof}
From~\eqref{eq2}, we have 
\begin{equation*}
J^w(Y_t)=-\frac{1}{2\overline F_X^2(\Phi^{-1}(t))}\int_t^{+\infty}x\frac{f_X^2(\Phi^{-1}(x))}{(\Phi'(\Phi^{-1}(x)))^2}\mathrm dx.
\end{equation*}
Let $\Phi$ be strictly increasing. Then, with a change of variable in the above integral, we get
\begin{equation*}
J^w(Y_t)=-\frac{1}{2\overline F_X^2(\Phi^{-1}(t))}\int_{\Phi^{-1}(t)}^{+\infty}\frac{\Phi(x)}{\Phi'(x)}f_X^2(x)\mathrm dx,
\end{equation*}
giving the first expression in~\eqref{eq4}. If $\Phi$ is strictly decreasing, we similarly obtain
\begin{equation*}
J^w(Y_t)=-\frac{1}{2F_X^2(\Phi^{-1}(t))}\int_0^{\Phi^{-1}(t)}\frac{\Phi(x)}{\left|\Phi'(x)\right|}f_X^2(x)\mathrm dx,
\end{equation*}
giving the second expression in~\eqref{eq4}. The proof of~\eqref{eq5} is quite similar.
\end{proof}

\section{Concluding remarks}
In this paper, some new measures of information have been introduced and studied. We have defined the weighted extropy as well as weighted residual and past extropies. We have presented some bounds and characterization results under monotonicity of hazard and reversed hazard functions. We have also presented bivariate versions of extropy and weighted extropy. We have presented numerous examples to illustrate all the concepts introduced here.

\section*{Acknowledgments}
Francesco Buono and Maria Longobardi are partially supported by the GNAMPA research group of INdAM (Istituto Nazionale di Alta Matematica) and MIUR-PRIN 2017, Project "Stochastic Models for Complex Systems" (No. 2017 JFFHSH).

\end{document}